\documentclass[11pt]{article}

\title{Relational Hypersequent $\mathbf{S4}$ and $\mathbf{B}$ are Cut-Free  Hypersequent Incomplete}
\author{Kai Tanter}
\date{}

\usepackage{graphicx}
\graphicspath{ {./images/} }

\usepackage{amsmath,amssymb,bussproofs, amsthm, fdsymbol}
\usepackage{float}
\usepackage{lscape}

\usepackage{tikz}
\usetikzlibrary{arrows,calc,patterns,positioning,shapes}
\usetikzlibrary{decorations.pathmorphing}

\tikzset{
      modal/.style={>=stealth,shorten >=1pt,shorten <=1pt,auto,node distance=1.5cm,semithick},
      world/.style={circle,draw,minimum size=1cm,fill=gray!15},
      point/.style={circle,draw,fill=black,inner sep=0.5mm},
      reflexive/.style={->,in=120,out=60,loop,looseness=#1},
      reflexive/.default={5},reflexive point/.style={->,in=135,out=45,loop,looseness=#1},
      reflexive point/.default={25},
}

\tikzset{
coil/.style={decorate, decoration={coil,amplitude=4pt,segment length=5pt}},
snake/.style={decorate, decoration={snake}},
zigzag/.style={decorate, decoration={zigzag}}
}

\theoremstyle{definition}
\newtheorem{definition}{Definition}[section]

\theoremstyle{lemma}
\newtheorem{lemma}{Lemma}[section]

\newtheorem{theorem}{Theorem}

\usepackage{caption}
\usepackage{subcaption}
\usepackage{cleveref}

\newcommand{\sslash}{\mathbin{/\mkern-6mu/}}

\begin{document}
\maketitle

\section{Introduction}

Andrew Parisi's relational hypersequent systems for standard modal logics $\mathbf{K}$ through to $\mathbf{S5}$ \cite{Parisi2017,ParisiForthcoming} are the first candidate hypersequent systems to meet two commonly cited criteria for ``good'' proof systems for modal logic: modularity and Do\v{s}en's Principle.\footnote{There are of course many other kinds of proof systems for modal logics. For example, display logics \cite{Wansing1998} and labelled sequent systems \cite{Negri2005}. See also \cref{fn:tree} for tree hypersequent systems.} Parisi's systems are intended to both provide the basis for an inferentialist account of modality and meet other criteria such as cut-admissibility and adequacy with regard to Kripke frames for these logics. In \cite{Parisi2017,ParisiForthcoming}, Parisi provides an indirect proof of sequent completeness for his systems via a translation to sequent systems for the respective modal logics. However, the proofs for his relational hypersequent $\mathbf{S4}$ and $\mathbf{B}$ require treating Cut as a basic rule. Samara Burns and Richard Zach \cite{BurnsForthcoming} have improved on these results by providing direct cut-free proofs of hypersequent completeness for relational hypersequent $\mathbf{K}$, $\mathbf{T}$ and $\mathbf{D}$, and Restall \cite{Restall2007,Restall2009} has done the same for a system equivalent to Parisi's relational hypersequent $\mathbf{S5}$. The current paper shows that Parisi's relational hypersequent $\mathbf{K4}$, $\mathbf{S4}$, $\mathbf{KB}$, and $\mathbf{B}$ are cut-free hypersequent incomplete, and that the former two are also cut-free sequent and formula incomplete, relative to standard Kripke frames for $\mathbf{K4}$, $\mathbf{S4}$, $\mathbf{KB}$ and $\mathbf{B}$ respectively. As a result, the systems fail to meet the criteria of cut-admissibility and adequacy with regard to Kripke frames. This leaves open the question of what hypersequent proof systems can meet Parisi's intended criteria and also of what kind of models Parisi's relational hypersequent $\mathbf{K4}$, $\mathbf{S4}$, $\mathbf{KB}$, and $\mathbf{B}$ are adequate in regard to.

We begin in \Cref{sec:background} by providing a brief overview of some of the criteria considered in the literature for ``good'' proof systems for modal logics. Next in sections \ref{sec:l}, \ref{sec:m} and \ref{sec:p} we respectively define the language, models and proof systems that will be studied in the paper. 
In \Cref{sec:RS4Incompleteness} we show that relational hypersequent $\mathbf{K4}$ and $\mathbf{S4}$ are cut-free  hypersequent, sequent and formula incomplete relative to Kripke frames for $\mathbf{S4}$ and $\mathbf{K4}$ respectively. In \Cref{sec:RBIncompleteness} we prove that relational hypersequent $\mathbf{KB}$ and $\mathbf{B}$ are cut-free hypersequent incomplete. We then contrast the two sets of results in \Cref{sec:contrast}, before \Cref{sec:end} concludes with a brief discussion of remaining open questions and consequences of these results.

\section{Background}
\label{sec:background}

Despite its axiomatic origins, contemporary work in modal logic is overwhelmingly model-theoretic. Work in the proof theory of modal logic has been focused on developing proof systems that are both adequate for different classes of models and which have proof-theoretically desirable properties. This work has various motivations, one of which is to provide the basis for an inferentialist account of modality.
Inferentialism is a theory of meaning which claims that meaning is determined by norms governing the use of expressions.\footnote{See \cite{Ballarin2021} for an overview of the development of modern modal logic; \cite{Brandom2000,Steinberger2017} for an overview of inferentialism; and \cite{Negri2011,Poggiolesi2012} for accessible overviews of work in the proof theory of modal logics. \cite{Parisi2017,ParisiForthcoming, Poggiolesi2011} also contains a discussion of and references to existing literature. } There is a natural fit between inferentialism and proof-theoretic approaches to semantics, as a proof theory can be interpreted as a formal representation of norms governing the use of expressions in a given language. One way for inferentialists to account for modality is to construct proof systems for modal logics that can be interpreted as determining the meaning of modals expressions like necessity $\Box$ and possibility $\lozenge$. Parisi \cite{Parisi2017,ParisiForthcoming} uses this to motivate several criteria for proof systems for modal logics, two of which are particularly relevant for the current paper: 
\begin{itemize}
\item Cut Admissibility: in a sequent calculus or similar setting, the resulting logic from the cut-free calculus is identical to that from the calculus with $\mathrm{Cut}$ as a basic rule;\footnote{In a natural deduction setting this would be the requirement that the system normalises.} and
\item Do\v{s}en's Principle: this principle applies to a set of calculi for modal logics and holds when the operational rules are shared, with calculi only differing in their structural rules. 
\end{itemize}
One way to think about Do\v{s}en's Principle is as a proof-theoretic analog of the way in which standard Kripke models for different modal logics share truth conditions for connectives but differ in the restrictions placed on the accessibility relation. A related but different criterion is that of modularity. Burns and Zach state it in a form directly relating proof systems to Kripke models:
\begin{itemize}
\item Modularity: ``each property of the accessibility relation [of a Kripke model] is captured by a single rule or set of rules.'' \cite[p.2]{BurnsForthcoming}
\end{itemize}
These properties are often considered desirable but are by no means universally endorsed. For example, see \cite[Chapter 1.10]{Poggiolesi2011} for an an argument against accepting Do\v{s}en's Principle.\footnote{\label{fn:tree}
\AxiomC{$\phi,\Gamma\Rightarrow\Delta$}
\RightLabel{\scriptsize{$\mathbf{t}$}}
\UnaryInfC{$\Box\phi,\Gamma\Rightarrow\Delta$}
\DisplayProof Poggiolesi's tree hypersequent systems \cite{Poggiolesi2011} are also motivated by inferentialism. In rejecting Do\v{s}en's Principle, Poggiolesi states the principle slightly differently and also appears to be working with a different distinction between operational (logical) and structural rules. For example, the rule $\mathbf{t}$ is classified as a structural rule despite it essentially involving $\Box$ in the conclusion sequent \cite[\S1.10]{BurnsForthcoming}.  
This differs from standard structural rules like weakening and contraction, which do not essentially involve any particular vocabulary, a feature that Parisi's structural rules do have. Do\v{s}en's Principle, as stated above, does not hold for Poggiolesi's systems because one system is obtained from another by varying both the structural and operational rules.}

\subsection{Language}
\label{sec:l}

\theoremstyle{definition}
\begin{definition}[$\mathcal{L}$]
\label{def:L}
$\mathcal{L}$ is the language made up of denumerably many atomic formula $p, q,...$, the unary connectives $\neg$ and $\Box$, and the binary connectives $\land$ and $\lor$, and whose sentences are all and only those generated recursively from the following rule: all atomic formulas $p$ are sentences and if $\phi$ and $\psi$ are sentences then so are $\neg\phi$, $\Box\phi$, $\phi\land\psi$ and $\phi\lor\psi$.\\
We will use lower case Greek as sentence variables and upper case for sets of sentences.
\end{definition}

\theoremstyle{definition}
\begin{definition}[Sequents and Hypersequents]
A \textit{sequent} $\Gamma \Rightarrow \Delta$ is an ordered pair of finite sets of sentences, with the turnstile $\Rightarrow$ separating each member of the pair. $\Gamma,\phi$ will be written as shorthand for $\Gamma\cup\{\phi\}$. Instead of writing the empty set $\emptyset$ we simply leave the relevant side of the turnstile blank. $S$, possibly subscripted, is used to represent arbitrary sequents in the metalanguage. \\
A \textit{hypersequent} $S_1\sslash...{\sslash}S_n$ is a finite sequence (list) of sequents, with $\sslash$ separating each member of the sequence. $G$ and $H$, possibly ``primed'' , are used to represent arbitrary hypersequents in our language. 

\end{definition}

\subsection{Models}
\label{sec:m}

\theoremstyle{definition}
\begin{definition}[Frames and Models]
\label{def:Kframe}
A Kripke frame $\mathfrak{F}$ is a pair $\langle W, R \rangle$ of points $W$ and a  binary relation $R$ on $W$.\\
 A Kripke model $\mathfrak{M}$ is a triple $\langle W, R, v, \rangle$ in which $\langle W,R\rangle$ is a Kripke frame and $v$ is a valuation function from members of $W$ and sentences of $\mathcal{L}$ to the truth values $\{1,0\}$.\\
We restrict $v$ as follows:
\begin{itemize}
\item $v(\neg\phi,x)= 1$ iff $v(\phi,x)=0$;
\item $v(\phi\land\psi,x)=1$ off $v(\phi,x)=1$ and $v(\psi,x)=1$; 
\item $v(\phi\lor\psi,x)=1$ off $v(\phi,x)=1$ or $v(\psi,x)=1$; and
\item $v(\Box\phi,x)=1$ iff for all $y$, if $xRy$ then $v(\phi,y)=1$. 
\end{itemize}
The conditions for $\neg$, $\land$ and $\lor$ are as in Boolean valuations but relative to a point.
\end{definition}

\noindent We can obtain classes of Kripke frames (models) for various modal logics by placing restrictions on $R$.
\theoremstyle{definition}
\begin{definition}[Branch of points]
A \textit{branch of points} $w_1,...,w_n$ is a sequence of points in a frame $\mathfrak{F}$ such that $w_iRw_{i+1}$ for all $i: 1\leq i \leq n-1$.
\end{definition}

\theoremstyle{definition}
\begin{definition}[Countermodel]
\label{def:counterexample}
\textit{Sequents}: 
A model $\mathfrak{M}$ is a \textit{countermodel} to a sequent $\Gamma\Rightarrow\Delta$ at a point $w$ iff for all $\phi\in\Gamma, v(\phi,w)=1$ and for all $\psi\in\Delta, v(\psi,w)=0$.\\
\textit{Hypersequents}:
A model $\mathfrak{M}$ is a \textit{countermodel} to a hypersequent $G$ iff there is a branch of points $w_1,...,w_n$ in $\mathfrak{M}$ such that $\mathfrak{M}$ is a countermodel to each sequent $S_i\in H$ at $w_i$ for all $i$:  $1\leq i\leq n$. \\
We write $\nvDash_XH$ to mean that a particular hypersequent $H$ has a countermodel in the class of $X$ frames, and $\vDash_X H$ to mean that a particular hypersequent $H$ has no countermodel, i.e.\ is valid, in the class of $X$ frames.

\end{definition}

\subsection{Proofs}
\label{sec:p}

\theoremstyle{definition}
\begin{definition}[The Hypersequent Calculus $\mathrm{RK}$]
A derivation in $\mathrm{RK}$ is a tree all of whose leaves are instances of the axiom $\mathrm{Id}$ and each non-leaf node is obtained from the nodes above via one of the rules of $\mathrm{RK}$ (see \Cref{fig:RK}).

\begin{figure}[H]

\caption{Rules of $\mathrm{RK}$}
\label{fig:RK}
\begin{tabular}{lcl}
\\
\AxiomC{}
\RightLabel{\scriptsize{$\mathrm{Id}$}}
\UnaryInfC{$p\Rightarrow p$}
\DisplayProof

&

\AxiomC{$G\sslash \Gamma\Rightarrow \phi,\Delta\sslash H$}
\AxiomC{$G\sslash \Gamma, \phi\Rightarrow \Delta\sslash H$}
\RightLabel{\scriptsize{$\mathrm{Cut}$}}
\BinaryInfC{$G\sslash \Gamma\Rightarrow \Delta\sslash H$}
\DisplayProof

 &

\AxiomC{$G$}
\RightLabel{\scriptsize{$EW\mathrm{L}$}}
\UnaryInfC{$\Rightarrow\sslash G$}
\DisplayProof
\\\\

\AxiomC{$G$}
\RightLabel{\scriptsize{$EW\mathrm{R}$}}
\UnaryInfC{$G\sslash \Rightarrow$}
\DisplayProof

&
\AxiomC{$G\sslash \Gamma\Rightarrow \Delta\sslash H$}
\RightLabel{\scriptsize{$T \mathrm{L}$}}
\UnaryInfC{$G\sslash \Gamma,\phi \Rightarrow \Delta\sslash H$}
\DisplayProof

&

\AxiomC{$G\sslash \Gamma\Rightarrow \Delta\sslash H$}
\RightLabel{\scriptsize{$T \mathrm{R}$}}
\UnaryInfC{$G\sslash \Gamma \Rightarrow\phi,\Delta\sslash H$}
\DisplayProof
\\\\
\end{tabular}

\begin{tabular}{lllc}

\AxiomC{$G\sslash\Gamma\Rightarrow\Delta\sslash \Rightarrow \phi$}
\RightLabel{\scriptsize{$\Box \mathrm{R}$}}
\UnaryInfC{$G\sslash\Gamma\Rightarrow \Box\phi,\Delta$}
\DisplayProof

&
\AxiomC{$G \sslash \Gamma\Rightarrow\Delta\sslash \Sigma,\phi\Rightarrow \Lambda \sslash H$}
\RightLabel{\scriptsize{$\Box \mathrm{L}$}}
\UnaryInfC{$G\sslash  \Gamma,\Box\phi\Rightarrow\Delta\sslash\Sigma\Rightarrow \Lambda\sslash H$}
\DisplayProof

&

\AxiomC{$G\sslash \Gamma\Rightarrow \Delta,\phi\sslash H$}
\RightLabel{\scriptsize{$\neg \mathrm{L}$}}
\UnaryInfC{$G\sslash \Gamma,\neg\phi \Rightarrow \Delta\sslash H$}
\DisplayProof

\\\\

\AxiomC{$G\sslash \Gamma,\phi\Rightarrow \Delta\sslash H$}
\RightLabel{\scriptsize{$\neg \mathrm{R}$}}
\UnaryInfC{$G\sslash \Gamma \Rightarrow\neg\phi,\Delta\sslash H$}
\DisplayProof

&

\AxiomC{$G\sslash \Gamma,\phi\Rightarrow \Delta\sslash H$}
\RightLabel{\scriptsize{$\land \mathrm{L_1}$}}
\UnaryInfC{$G\sslash \Gamma,\phi\land\psi \Rightarrow \Delta\sslash H$}
\DisplayProof
&
\AxiomC{$G\sslash \Gamma,\psi\Rightarrow \Delta\sslash H$}
\RightLabel{\scriptsize{$\land \mathrm{L_2}$}}
\UnaryInfC{$G\sslash \Gamma,\phi\land\psi \Rightarrow \Delta\sslash H$}
\DisplayProof
\\\\
\end{tabular}

\begin{tabular}{llccccccccccc}

\AxiomC{$G\sslash \Gamma\Rightarrow \phi,\Delta\sslash H$}
\AxiomC{$G\sslash \Gamma\Rightarrow \psi,\Delta\sslash H$}
\RightLabel{\scriptsize{$\land\mathrm{R}$}}
\BinaryInfC{$G\sslash \Gamma\Rightarrow \phi\land\psi,\Delta\sslash H$}
\DisplayProof

&

\AxiomC{$G\sslash \Gamma\Rightarrow \Delta,\phi\sslash H$}
\RightLabel{\scriptsize{$\lor \mathrm{R_1}$}}
\UnaryInfC{$G\sslash \Gamma,\phi\lor\psi \Rightarrow \Delta\sslash H$}
\DisplayProof

\\\\

\AxiomC{$G\sslash \Gamma, \phi\Rightarrow\Delta\sslash H$}
\AxiomC{$G\sslash \Gamma, \psi\Rightarrow \Delta\sslash H$}
\RightLabel{\scriptsize{$\lor\mathrm{L}$}}
\BinaryInfC{$G\sslash \Gamma,\phi\lor\psi\Rightarrow \Delta\sslash H$}
\DisplayProof

&

\AxiomC{$G\sslash \Gamma\Rightarrow \Delta,\psi\sslash H$}
\RightLabel{\scriptsize{$\lor \mathrm{R_2}$}}
\UnaryInfC{$G\sslash \Gamma,\phi\lor\psi \Rightarrow \Delta\sslash H$}
\DisplayProof

\end{tabular}

\end{figure}
\noindent Additional systems are obtained from $\mathrm{RK}$ by the addition of further structural rules from \Cref{fig:add}, as set out in \Cref{fig:syst}.

\begin{figure}[H]
\caption{Additional Structural Rules}
\label{fig:add}
\begin{tabular}{llll}
\\
 \AxiomC{$G\sslash\Gamma\Rightarrow\Delta\sslash\Gamma\Rightarrow\Delta\sslash H$}
\RightLabel{\scriptsize{$EC$}}
\UnaryInfC{$G\sslash \Gamma\Rightarrow\Delta\sslash H$}
\DisplayProof
&

 \AxiomC{$\Gamma_1\Rightarrow\Delta_1\sslash...\sslash\Gamma_n\Rightarrow\Delta_n$}
\RightLabel{\scriptsize{$Sym$}}
\UnaryInfC{$\Gamma_n\Rightarrow\Delta_n\sslash...\sslash\Gamma_1\Rightarrow\Delta_1$}
\DisplayProof

&
\AxiomC{$G\sslash H$}
\RightLabel{\scriptsize{$EW$}}
\UnaryInfC{$G\sslash \Rightarrow\sslash H$}
\DisplayProof
\\
\end{tabular}
\begin{tabular}{rrrrrrrrrrrrrrrrrrrrrrr}
\\
&&&&&&&
 \AxiomC{$G\sslash\Gamma\Rightarrow\Delta\sslash\Sigma\Rightarrow\Lambda\sslash H$}
\RightLabel{\scriptsize{$EE$}}
\UnaryInfC{$G\sslash\Sigma\Rightarrow\Lambda\sslash\Gamma\Rightarrow\Delta\sslash H$}
\DisplayProof

&

\AxiomC{$G\sslash\Rightarrow$}
\RightLabel{\scriptsize{$Drop$}}
\UnaryInfC{$G$}
\DisplayProof

\end{tabular}

\end{figure}

\begin{figure}[H]
\caption{Hypersequent Systems, Logics and Frame Conditions}
\label{fig:syst}
\begin{tabular}{llll}
\textbf{System} & \textbf{Additional Rules} & \textbf{Intended Logic.} & \textbf{Intended Frame Conditions}\\
$\mathrm{RD}$ & $Drop$ & $\mathbf{D}$ & {Seriality}\\
$\mathrm{RT}$ & $EC$ & $\mathbf{T}$ & {Reflexivity} \\
$\mathrm{RKB}$ & $Sym$ & $\mathbf{KB}$ & {Symmetry}\\
$\mathrm{RK4}$ & $EW$ & $\mathbf{K4}$ & {Transitivity}\\
$\mathrm{RB}$ & $EC$ and $Sym$ & $\mathbf{B}$ & {Reflexivity} and {Symmetry}\\
$\mathrm{RS4}$ & $EC$ and $EW$ & $\mathbf{S4}$ & {Reflexivity} and {Transitivity}\\
$\mathrm{RS5}$ & $EC$, $EW$ and $EE$ & $\mathbf{S5}$ & {Reflexivity}, {Symmetry} and {Transitivity}\\

\end{tabular}
\end{figure}
\noindent To say that a particular hypersequent $H$ has a derivation in a particular system $\mathrm{RX}$, we write $\vdash_{\mathrm{RX}}H$. $\nvdash_{\mathrm{RX}}H$ means that the hypersequent $H$ has no derivation in $\mathrm{RX}$. We write $\vdash_{\mathrm{RX_{CF}}} H$ to say that $H$ has a cut-free derivation in $\mathrm{RX}$. 

Burns and Zach identify Parisi's relational hypersequent systems as ``the first candidates for hypersequent calculi for modal logics that are both modular and conform to Do\v{s}en's principle'' \cite[p.2]{BurnsForthcoming}.\footnote{The systems presented in this paper are strictly speaking Burns and Zach's. Parisi's lack the $EW\mathrm{R}$ rule that is needed for hypersequent completeness and Parisi refers to his systems with the prefix `$\mathrm{H}$' rather than `$\mathrm{R}$'. Burns and Zach, and this paper, also use Lellman's notation from \cite{Lellmann2015}. Disjunction $\lor$ and its corresponding rules have been added to make $C$, the counterexample to $\mathrm{RK4}$ and $\mathrm{RS4}$'s completeness, more perspicuous. This is only needed to make the proof simpler, as disjunction can be defined using conjunction and negation as usual.}  However, as they note, the systems are not completely modular as $\mathrm{RS5}$ is obtained by adding $EE$ to $\mathrm{RS4}$ rather than adding $Sym$. The fact that $\mathrm{RS4}$ and $\mathrm{RB}$ turn out to be cut-free hypersequent incomplete is plausibly connected to this lack of modularity.\footnote{Interestingly, this lack of modularity in the move from $\mathbf{S4}$ to $\mathbf{S5}$ is shared by Poggiolesi's tree hypersequent systems \cite[p.125-6]{Poggiolesi2011}.} 

\end{definition}

\subsection{Completeness}
\label{sec:c}

Parisi's systems are sound relative to standard Kripke frames, in the sense that whenever a hypersequent is provable, there is no counterexample \cite{ParisiForthcoming}. The converse of this, completeness, is the focus of the paper.  

\begin{definition}[Completeness]
A hypersequent calculus $\mathrm{RX}$ is \emph{Y-Complete} relative to a class of frames $\mathbb{S}$ iff whenever $\vDash_{\mathbb{S}}Y$ then $\vdash_{\mathrm{RX}}Y$. When $Y$ stands in for: arbitrary hypersequents we say that $\mathrm{RX}$ is \emph{Hypersequent-Complete} (H-Complete); hypersequents of the form $\Gamma\Rightarrow\Delta$ that $\mathrm{RX}$ is \emph{Sequent-Complete} (S-Complete); and hypersequents of the form $\Rightarrow\phi$ that $\mathrm{RX}$ is \emph{Formula-Complete} (F-Complete).\\
A hypersequent calculus $\mathrm{RX}$ is \emph{Cut-Free Y-Complete} (CF Y-Complete) relative to a class of frames $\mathbb{S}$ iff whenever $\vDash_{\mathbb{S}}Y$ then $\vdash_{\mathrm{RX_{CF}}}Y$.

\end{definition} 

\begin{figure}[h]
\caption{State of Play}
\begin{small}
\begin{tabular}{cccccc}
\textbf{System}  &  \textbf{S-Com} & \textbf{ H-Com} & \textbf{CF F-Com} & \textbf{CF S-Com} & \textbf{CF H-Com}\\
RK  & Y (P)  & Y (B\&Z) & Y (P) & Y (P)  & Y (B\&Z) \\
RD  & Y (P) & Y (B\&Z) & Y (P) &Y (P)  & Y (B\&Z) \\
RT   & Y (P) & Y (B\&Z) & Y (P) &Y (P)  & Y (B\&Z) \\
RKB   & Y (P)  &? & ? & ? & \textbf{\emph{N}}\\
RB   & Y (P)  &? & ? & ? & \textbf{\emph{N}}\\
RK4  & Y (P)  &\textbf{\emph{Y}} & \textbf{\emph{N}} & \textbf{\emph{N}} & \textbf{\emph{N}}\\
RS4  & Y (P)  &\textbf{\emph{Y}} & \textbf{\emph{N}} & \textbf{\emph{N}} & \textbf{\emph{N}}\\
RS5  & Y (P) & Y (P)  & Y (P) & Y (P)  & Y(P)
\end{tabular}
\end{small}
\label{fig:SP}

\end{figure}

\Cref{fig:SP} contains the current state of play when it comes to completeness results for Parisi's relational hypersequent systems for modal logic, leaving out Formula-Completeness as a distinct category.
Parisi has proved sequent-completeness for all his systems, in some cases cut-free and in others only for the system with $\mathrm{Cut}$ as a basic rule.\footnote{Parisi's proves sequent completeness for $\mathrm{RK}$, $\mathrm{RD}$ and $\mathrm{RT}$ without using $\mathrm{Cut}$ (hence cut-free). The proof of sequent completeness of $\mathrm{RS5}$ does use $\mathrm{Cut}$, as does that for $\mathrm{RKB}$ and $\mathrm{RS4}$. These proofs work by showing that his systems are sequent equivalent to a sequent system for $\mathbf{K}$, $\mathbf{D}$, $\mathbf{T}$ and $\mathbf{S5}$ respectively, which are already known to be sequent complete.
For $\mathrm{RS5}$, he then shows that $\mathrm{Cut}$ is an admissible rule in $\mathrm{RS5}$ by showing that it is hypersequent equivalent to a cut-free $\mathbf{S5}$ hypersequent system of Restall's. Cut-free sequent completeness then follows from cut-free hypersequent completeness.  } Burns and Zach have shown direct cut-free hypersequent completeness for $\mathrm{RK}$, $\mathrm{RD}$ and $\mathrm{RT}$ (from this it follows that these systems are cut-free sequent complete and hypersequent complete). The cells marked \textbf{\emph{N}} are answered in the negative in the current paper: $\mathrm{RK4}$, $\mathrm{RS4}$, $\mathrm{RKB}$ and $\mathrm{RB}$ are all cut-free hypersequent incomplete, and the former two are also cut-free sequent and formula incomplete. Sequent and formula completeness remains open for $\mathrm{RKB}$ and $\mathrm{RB}$. The cells marked \textbf{\emph{Y}} are answered positively in Appendix A: $\mathrm{RK4}$ and $\mathrm{RS4}$ with $\mathrm{Cut}$ as a basic rule are not only sequent complete, but also hypersequent complete. 

\section{RK4 and RS4 Cut-Free Incompleteness}
\label{sec:RS4Incompleteness}

In this section we prove that $\mathrm{RK4}$ and $\mathrm{RS4}$ are cut-free hypersequent, sequent and formula incomplete. The outline of the proof is as follows: first, we identify a hypersequent $C$ of the form $\Rightarrow\phi$ that is $\mathbf{K4}$ and $\mathbf{S4}$ valid; second; we define a class of models, $\mathrm{PS4}$ models, relative to which $C$ is invalid; third, we show that cut-free $\mathrm{RK4}$ and $\mathrm{RS4}$ are sound relative to $\mathrm{PS4}$ models, resulting in both of $\mathrm{RK4}$ and $\mathrm{RS4}$ being cut-free formula incomplete; fourth, as an immediate consequence, cut-free $\mathrm{RK4}$ and $\mathrm{RS4}$ are both also sequent and hypersequent incomplete. 

To begin the proof, we show that $C= \Rightarrow \neg\Box\neg\Box(p\land q)\lor\Box(\neg\Box p\lor\Box\neg\Box q)$ is $\mathbf{K4}$ and $\mathbf{S4}$ valid.

\theoremstyle{lemma}
\begin{lemma}
\label{lem:K4valid}

$\vDash_{\mathbf{K4}} C= \Rightarrow{\neg\Box}{\neg\Box}(p\land q)\lor\Box({\neg\Box} p\lor{\Box\neg}{\Box} q)$

\end{lemma}

\begin{proof}
The proof is a \textit{reductio} of the assumption that there is a countermodel. 
\begin{figure}[h]
\caption{No $\mathbf{K4}$ Countermodel to $C$}

\flushleft
\begin{tikzpicture}[modal,world/.append style={minimum size=1.2cm}]

\node[world] (i) [label=below:{$i$}] {${\Box\neg\Box}(p\land q)$};

\node[world] (j) [label=below:{$j$},right=of i] {$\Box p$} ;

\node[world] (k) [label=below:{$k$},right=of j] {$\Box q$} ;

\node[world] (m) [label=below:{$m$},right=of k] {$\neg(p\land q)$,
$p\land q$};

\path[->] (i) edge (j); \path[->] (j) edge (k); \path[->] (k) edge (m);
\path[->] (i) edge[bend left=60] (k);
\path[->] (j) edge[bend left=60] (m);

\end{tikzpicture}

\end{figure}
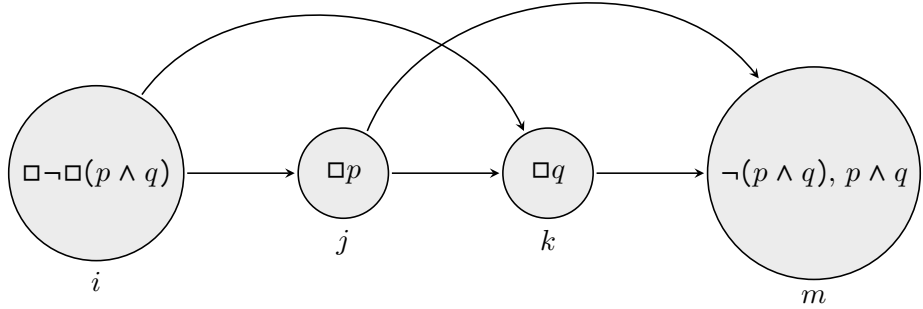
Once the $\lor$s and $\Box$s have been decomposed, for there to be a countermodel, we need three points, call them $i$, $j$ and $k$, where $v(\Box\neg\Box(p\land q),i)=1=v(\Box p,j)=v(\Box q,k)=1$, $iRj$ and $jRk$. It follows that $iRk$ and $v({\neg\Box}(p\land q),k)=1$. So there must be a point $m$, where $kRm$, $v(\neg(p\land q),m)=1=v(q,m)$. But then $jRm$ also, and $v(p,m)=1=v(p\land q,m)\neq 0=v(p\land q,m)$. 

\end{proof}

\theoremstyle{lemma}
\begin{lemma}
\label{lem:S4valid}
$\vDash_{\mathbf{S4}} C= \Rightarrow {{\neg\Box}\neg\Box}(p\land q)\lor\Box({\neg\Box} p\lor{\Box\neg\Box} q)$

\end{lemma}

\begin{proof}
Every $\mathbf{S4}$ model is also a $\mathbf{K4}$ model. So, if there were an $\mathbf{S4}$ countermodel to $C$, there would be a $\mathbf{K4}$ countermodel. From  \Cref{lem:K4valid}, there is no $\mathbf{K4}$ countermodel. Hence, there isn't an $\mathbf{S4}$ one either. 
\end{proof}

We next define a new class of models, $\mathrm{PS4}$ models; show that both $\mathrm{RK4}$ and $\mathrm{RS4}$ are sound relative to these models; and that there is a $\mathrm{PS4}$ model that is a countermodel to $C$. 

\theoremstyle{definition}
\begin{definition}[$\mathrm{PS4}$ Frames]
\label{def:PS4F}

A Pseudo $\mathrm{S4}$ ($\mathrm{PS4}$) frame is a triple $\langle W, R, S\rangle$ where $W$ is a non-empty set of points and both $R$ and $S$ are binary relations on W.\\ We set the following restrictions on $R$ and $S$, displayed in \Cref{Fig:RS}: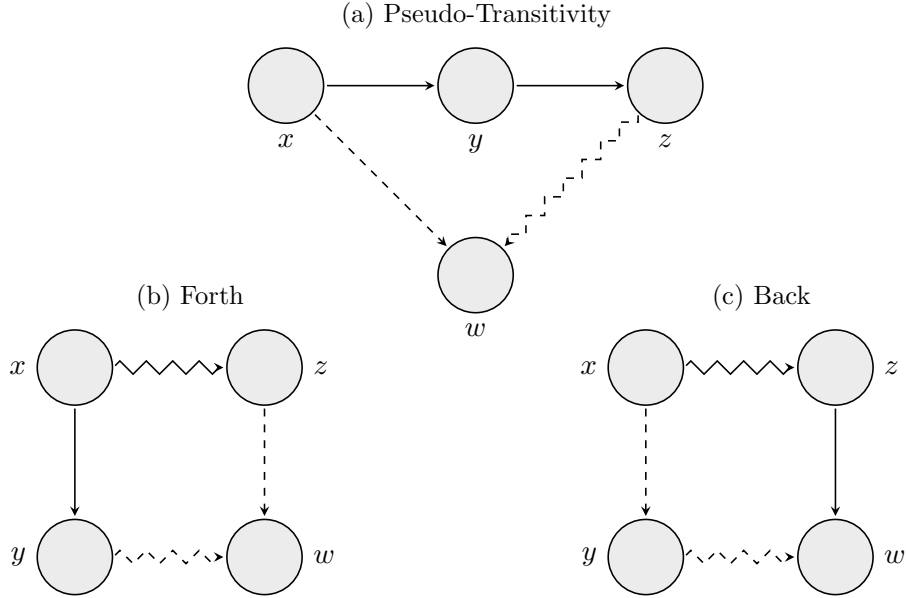
\begin{figure}
\caption{Additional Restrictions on $R$ (straight lines) and $S$ (zigzag lines)}
\label{Fig:RS}
\begin{subfigure}[c][4cm]{\textwidth}
\caption{Pseudo-Transitivity}
\centering
\begin{tikzpicture}[modal]

\node[world] (x) [label=below: {$x$}] {};

\node[world] (y) [label=below: {$y$},right=of x] {};

\node[world] (z) [label=below: {$z$},right=of y] {};

\node[world] (w) [label=below: {$w$},below=of y] {};

\path[->] (x) edge (y); \path[->] (y) edge (z); 

\path[->] (x) edge[dashed] (w);

\path[->] (z) edge[zigzag,dashed] (w);
\end{tikzpicture}
\end{subfigure}

\begin{subfigure}[c][3cm]{0.4\textwidth}
\caption{Forth}
\begin{tikzpicture}[modal]

\node[world] (x) [label=left: {$x$}] {};

\node[world] (y) [label=left: {$y$},below=of x] {};

\node[world] (z) [label=right: {$z$},right=of x] {};

\node[world] (w) [label=right: {$w$},below=of z] {};

\path[->] (x) edge (y); 
\path[->] (z) edge[dashed] (w); 

\path[->] (x) edge[zigzag] (z);

\path[->] (y) edge[zigzag,dashed] (w);
\end{tikzpicture}
\end{subfigure}
\hfill
\begin{subfigure}[c][3cm]{0.4\textwidth}
\caption{Back}
\begin{tikzpicture}[modal]

\node[world] (x) [label=left: {$x$}] {};

\node[world] (y) [label=left: {$y$},below=of x] {};

\node[world] (z) [label=right: {$z$},right=of x] {};

\node[world] (w) [label=right: {$w$},below=of z] {};

\path[->] (x) edge[dashed] (y); 
\path[->] (z) edge (w); 

\path[->] (x) edge[zigzag] (z);

\path[->] (y) edge[zigzag,dashed] (w);

\end{tikzpicture}
\end{subfigure}

\end{figure}
\begin{enumerate}
	\item \textit{S Reflexivity}: For all points $x$, $xSx$;
	\item \textit{R Reflexivity}: For all points $x$, $xRx$;

	\item	\textit{Pseudo-Transitivity}: For all points $x$, $y$, and $z$: if $xRy$ and $yRz$, then there is a $w$: $xRw$ and $zSw$;

	\item \textit{Forth}: For all points $x$, $y$, and $z$: if $xRy$ and $xSz$, then there is a point $w$: $zRw$ and $ySw$; 

	\item \textit{Back}: For all points $x$, $z$, and $w$: if $xSz$ and $zRw$, then there is a point $y$: $xRy$ and $ySw$.\footnote{Forth and Back are standard bisimulation conditions, sometimes knows as Zig and Zag respectively. See \cite[\S2.2]{Blackburn2001} for an accessible introduction to bisimulation in modal logic.} 
\end{enumerate}
\end{definition}

Having defined $\mathrm{PS4}$ frames, we will use the following definition of an information order on points in a model to then define $\mathrm{PS4}$ models in \Cref{def:PS4M}. 

\theoremstyle{definition}
\begin{definition}[Information Order]
	\label{def:Inf}
	Given two points in a model $x$ and $y$,  $x\sqsubseteq y$ iff for all atomics $p$, if $v(x,p)\in\{1,0\}$, then $v(y,p)=v(x,p)$. When $x	\sqsubseteq y$ we say that $x$ is earlier than $y$ in the information order.
\end{definition}

\theoremstyle{definition}
\begin{definition}[$\mathrm{PS4}$ Models]
\label{def:PS4M}

A Pseudo $\mathrm{S4}$ ($\mathrm{PS4}$) model is a quadruple $\langle W, R, S, v\rangle$ where $\langle W, R, S\rangle$ is a $\mathrm{PS4}$ frame and $v$ is a 
valuation function from pairs of a point and a formula to $\langle 1, *, 0\rangle$. We set the following restrictions 
on $v$:
\begin{enumerate}
	\item Strong Kleene: $v$ uses standard modal Strong Kleene truth conditions:
		\begin{itemize}
			\item [$\neg 1$:] $v(\neg\phi,x)=1$ iff $v(\phi,x)=0$;
			\item [$\neg 0$:] $v(\neg\phi,x)=0$ iff $v(\phi,x)=1$;
			\item [$\land 1$:] $v(\phi\land\psi,x)=1$ iff $v(\phi,x)=1$ and $v(x,\psi)=1$;
			\item [$\land 0$:] $v(\phi\land\psi,x)=0$ iff $v(\phi,x)=0$ or $v(x,\psi)=0$;
	             \item [$\lor 1$:] $v(\phi\land\psi,x)=1$ iff $v(\phi,x)=1$ or $v(\psi,x)=1$;
			\item [$\lor 0$:] $v(\phi\land\psi,x)=0$ iff $v(\phi,x)=0$ and $v(\psi,x)=0$;
			\item [$\Box 1$:] $v(\Box\phi,x)=1$ iff for all $y$: if $xRy$ then $v(\phi,y)=1$;
			\item [$\Box 0$:] $v(\Box\phi,x)=0$ iff there is a $y$: $xRy$ and $v(\phi,y)=0$.
		\end{itemize}	
Note that these match ``classical'' Kripke models from \Cref{def:Kframe} for 1 and 0, but leave a ``gap'' for *. 

	\item \textit{S Information Preservation} ($S_\sqsubseteq$): For all points $x$, and $y$ : if $xSy$, then $x \sqsubseteq y $.

\end{enumerate}
\end{definition}

The $S$ Information Preservation condition in \Cref{def:PS4M} means that $S$ preserves the information order. We now show that $S$ preserves the truth and falsity of formulae in general, rather than just atomics. This fact will be used in the proof of the soundness of $\mathrm{RK4}$ and $\mathrm{RS4}$ relative to $\mathrm{PS4}$ models. 

\theoremstyle{lemma}
\begin{lemma}
\label{lem:InfoPre}

If $x S y $, then for all formulae $\phi$, if $v(\phi,x)\in \{1,0\}$ then $v(\phi,y)=v(\phi,x)$

\end{lemma}

\begin{proof}

We prove this by induction on the complexity of $\phi$. For atomic sentences, it follows immediately from the $S$ preservation condition on $v$.\\
The extensional connectives are fairly simple.\footnote{See \cite[\S1.2, \S6.2]{Blamey2002} and \cite[p.49]{Muskens1995} for discussion.} We work through conjunction, leaving negation and disjunction to the reader. We have two subcases. Let $\phi=\psi\land\delta$:
\begin{enumerate}
\item $v(\psi\land\delta,x)=1$. It follows that $v(\psi,x)=1=v(\delta,x)$. So, by the induction hypothesis $v(\psi,y)=1=v(\delta,y)$. Hence, $v(\psi\land\delta,y)=1$ also.
\item $v(\psi\land\delta,x)=0$. It follows that either $v(\psi,x)=0$ or $v(\delta,x)=0$. So, by the induction hypothesis either $v(\psi,y)=0$ or $v(\delta,y)=0$ respectively. Hence, $v(\psi\land\delta,y)=0$ also.
\end{enumerate}
Necessity $\Box$ is the trickier case and here the bisimulation conditions play a role. We have two subcases. Let $\phi=\Box\psi$:
\begin{enumerate}
\item $v(\Box\psi,x)=1$. By assumption $xSy$. We need to show $v(\Box\psi,y)=1$. For this, we need to first show that for any $z$, if $yRz$ then $v(\psi,z)=1$. Suppose there is some such $z$. By $Back$, it follows that there is a $w$ such that $xRw$ and $wSz$. $v(\psi,w)=1$ and therefore by the induction hypothesis, $v(\psi,z)=1$ also. Hence, $v(\Box\psi,y)=1$. The condition holds.
\item $v(\Box\psi,x)=0$. Therefore, there is a $w$ such that $xRw$ and $v(\psi,w)=0$. By assumption $xSy$. We need to show $v(\Box\psi,y) = 0$. For this, we need to show that there is $z$ where $yRz$ and $v(\psi,z)=0$. By $Forth$, there is a $z$ where $yRz$ and $wSz$. It follows from the induction hypothesis that $v(\psi,z)=0$. Hence, $v(\Box\psi,y)=0$. The condition holds. 
\end{enumerate}

\end{proof}

The following lemma will also be used in the proof of $\mathrm{RK4}$ and $\mathrm{RS4}$'s soundness relative to $\mathrm{PS4}$ models, specifically for $EW$. 

\theoremstyle{lemma}
\begin{lemma}
\label{lem:branchpres}

If in a model there is a branch of points $w_1,...,w_{i-1},w_i,w_{i+1}...,w_n$ then in the same model there is a branch $w_1,...,w_{i-1},w_{i+1}',...,w_n'$, where for all $j$, $i+1\leq j \leq n, w_jSw_j'$.

\end{lemma}

\begin{proof}
Suppose there is a branch $w_1,...,w_{i-1},w_i,w_{i+1}...,w_n$. By Pseudo-Transitivity there must be a point $w_{i+1}'$ such that $w_{i-1}Rw_{i+1}'$ and $w_{i+1}Sw_{i+1}'$.  By $Forth$, for all $j$, $i+2\leq j\leq n$ such that $w_{j-1}Rw_j$ and $w_{j-1}Sw_{j-1}'$, then there is a $w_j'$ such that $w_{j-1}'Rw_j'$ and $w_jSw_j'$. $i+1$ iterations of this consequence of $Forth$ will result in the desired branch. 
\end{proof}

What \Cref{lem:branchpres} tells us is that whenever Pseudo-Transitivity requires that we make an informational ``copy'' of a point $z$, we also make an informational ``copy'' of each branch of points from $z$ onwards. This is will be essential for the soundness of the $EW$ rule in the following \cref{lem:PS4sound}.

\theoremstyle{lemma}
\begin{lemma}
\label{lem:PS4sound}

If $\vdash_{RS4_{CF}} H$ then $\vDash_{PS4} H$

\end{lemma}

\begin{proof}

The proof proceeds by induction on the length of derivations. Much of this proof is routine. We only explicitly consider $EW$. The rest are the same as in Parisi \cite{ParisiForthcoming}. 

\begin{centering}
\begin{tabular}{c}

\AxiomC{$\vdots$}
\noLine
\UnaryInfC{$G\sslash H$}
\RightLabel{\scriptsize{$EW$}}
\UnaryInfC{$G\sslash \Rightarrow\sslash H$}
\DisplayProof\\\\
\end{tabular}\\
\end{centering}
\noindent Suppose there was a countermodel to the endhypersequent. This would be a model with a branch $w_1,...,w_{i-1},w_i,w_{i+1},...,w_n$, where $w_i$ countermodels $\Rightarrow$. By \Cref{lem:branchpres} there is a branch  $w_1,...,w_{i-1},w_{i+1}',...,w_n'$, where for all $j$, $i+1\leq j \leq n, w_jSw_j'$. By \Cref{lem:InfoPre}, for all $j, i+1\leq j \leq n$, for all $\phi$,  if $v(\phi,w_j)\in \{1,0\}$ then $v(\phi,w_j')=v(\phi,w_j)$. So, the branch $w_1,...,w_{i-1},w_{i+1}',...,w_n'$ is a countermodel to the premise hypersequent.
\end{proof}

\begin{lemma}
\label{lem:PK4sound}
If $\vdash_{\mathrm{RK4_{CF}}} H$ then $\vDash_{\mathrm{PS4}} H$

\end{lemma}

\begin{proof}
$\mathrm{RK4}$ differs from $\mathrm{RS4}$ only in lacking the external structural rule $EC$. So the proof proceeds as for \Cref{lem:PS4sound} above but without the $EC$ case. 
\end{proof}

Given lemmas \ref{lem:PS4sound} and \ref{lem:PK4sound}, if there is a $\mathrm{PS4}$ counterexample to a hypersequent, it will be unprovable in each of $\mathrm{RS4}$ and $\mathrm{RK4}$. We now show that there is a $\mathrm{PS4}$ counterexample to $C$. 

\begin{lemma}
\label{lem:PS4invalid}
$\nvDash_{\mathrm{PS4}}  C= \Rightarrow {\neg\Box}{\neg\Box}(p\land q)\lor\Box({\neg\Box} p\lor{\Box\neg}{\Box} q)$
\end{lemma}

\begin{proof}
The model in \Cref{fig:PS4countermodel} is a countermodel to $C$.
\begin{figure}[h]
\caption{$\mathrm{PS4}$ countermodel to $C$}
\label{fig:PS4countermodel}
\begin{tikzpicture}[modal,world/.append style={minimum size=1.5cm}]

\node[world] (i) [label=above:{$i: {\Box\neg{\Box(p\land q)}}$}] {$\neg q,\neg p$};

\node[world] (j) [label=above:{$j:  {\Box p}$},right=of i] {$\neg q, p$} ;

\node[world] (k) [label=above:{$k: {\Box q}$},right=of j] {$p,q$} ;

\node[world] (m) [label=right:{$m$},right=of k] {$q,*p$};

\node[world] (n) [label=below:{$n: {\Box q}$},below=of k] {$p,q$} ;

\node[world] (l) [label=below:{$l$},below=of j] {$\neg p,q$} ;

\path[->] (i) edge (j); \path[->] (j) edge (k); \path[->] (k) edge (m);
\path[->] (i) edge (w);
\path[->] (i) edge (n);
\path[->] (n) edge (l);

\path[->] (m) edge[bend right=55, zigzag] (k);
\path[->] (m) edge[zigzag] (n);
\path[->] (m) edge[bend left=100, zigzag] (l);
\path[->] (k) edge[zigzag] (n);
\end{tikzpicture}
\end{figure}
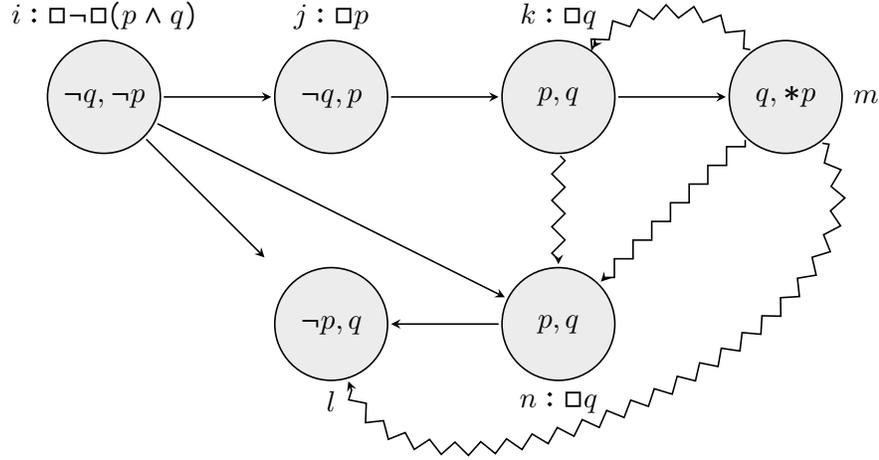
We have  a $PS4$ frame with:
\begin{itemize}
\item six points $i, j, k, m, n, l$;
\item in addition to $R$ being reflexive, we have $iRk, jRk, kRm, iRn, nRl, iRl$; 
\item in addition to $S$ being reflexive, we have $mSk, mSn, mSl, kSn$.
\end{itemize}
We set the valuation function $v$ such that: 
\begin{itemize}
\item $v(q,i)=v(p,i)=v(q,j)=v(p,l)=0$; 
\item $v(p,j)=v(p,k)=v(q,k)=v(q,m)=v(p,n)=v(q,n)=v(q,l)=1$;  
\item $v(p,m)=*$.
\end{itemize}
Verification that this is indeed both a $\mathrm{PS4}$ frame and a $\mathrm{PS4}$ model, and that it is a countermodel to $C$ at the branch made up of the single point $i$ is left to the reader. For the verification that it is a countermodel, first identify that $v(\Box\neg\Box(p\land q),i)=1=v(\Box p,j)=v(\Box q,k)$. The countermodel works by having $iRn$ instead of $iRk$, where $k\sqsubseteq n$, and $v(\Box(p\land q),n)=0$. The branch $i,n,l$ doesn't contain $j$ and so we can have $v(\Box p,n)=1$ but $v(\Box q,n)=0$. It is helpful to compare this with the reasoning in the proof of \Cref{lem:K4valid}.
\end{proof}

We now have what we need to show that $\mathrm{RS4}$ is cut-free incomplete. 

\begin{theorem}
\label{th:RS4hinc}
$\mathrm{RS4}$ is cut-free (i) formula, (ii) sequent, and (iii) hypersequent incomplete relative to $\mathrm{S4}$ (transitive and reflexive Kripke) frames.
\end{theorem}

\begin{proof}
(i) follows from lemmas \ref{lem:S4valid}, \ref{lem:PS4sound}, and \ref{lem:PS4invalid}.\\
From \Cref{lem:PS4sound} we know that if $C$ were $\mathrm{RS4}$ cut-free provable, then $C$ would be valid in $\mathrm{PS4}$ models. However, from this and \Cref{lem:PS4invalid}, we know that $C$ is not cut-free provable in RS4. Yet $C$ is valid in $\mathbf{S4}$ Kripke frames. So, $\mathrm{RS4}$ is cut-free incomplete relative to $\mathbf{S4}$ Kripke frames.\\
(ii) and (iiii) follow immediately from (i).

\end{proof}

\begin{theorem}
\label{th:RK4hinc}
$\mathrm{RK4}$ is cut-free (i) formula, (ii) sequent, and (iii) hypersequent incomplete relative to $\mathbf{S4}$ (transitive and reflexive Kripke) frames.
\end{theorem}

\begin{proof}
The reasoning is the same as for \Cref{th:RS4hinc} but using lemmas \ref{lem:K4valid}, \ref{lem:PK4sound} and \ref{lem:PS4invalid}.
\end{proof}

\section{RKB and RB Cut-Free Incompleteness}
\label{sec:RBIncompleteness}

The proof in this section has a slightly different structure to that of the previous section. We first identify a hypersequent $J$ that is $\mathbf{KB}$ and $\mathbf{B}$ valid; second, we define a new proof system $\mathrm{RTB}$ and show that $J$ is unprovable in both $\mathrm{RKB}$ and $\mathrm{RTB}$; third, we show that the rule $EC$ is admissable in $\mathrm{RTB}$, meaning that anything that is $\mathrm{RB}$ provable is also $\mathrm{RTB}$ provable. It follows that $J$ is also unprovable in $\mathrm{RB}$, resulting in both $\mathrm{RKB}$ and $\mathrm{RB}$ being hypersequent cut-free incomplete. The cut-free formula and sequent completeness of the two systems, however, remains open. 

\begin{lemma}
\label{lem:{KB}valid}
	$\vDash_{\mathbf{KB}} J=\Rightarrow p\sslash\Rightarrow \Box(\neg{\Box\Box} p\land \neg{\Box\Box} q)\sslash\Rightarrow q$.
\end{lemma}

\begin{proof}
The proof is a \textit{reductio} of the assumption that there is a countermodel. We start with three points, call them $i$, $j$ and $k$, where $iRj$, $jRk$, $v(p,i)=0=v(q,j)=v(\Box(\neg{\Box\Box}p\land\neg{\Box\Box}q),j)$. Because this is a symmetric Kripke frame, $jRi$ and $kRj$ also. Because $v(\Box(\neg{\Box\Box} p\land\neg{\Box\Box} q),j)=0$ there must be a point, call it $m$, where $jRm$, $mRj$ and $v(\neg{\Box\Box} p\land\neg{\Box\Box} q,m)=0$. 
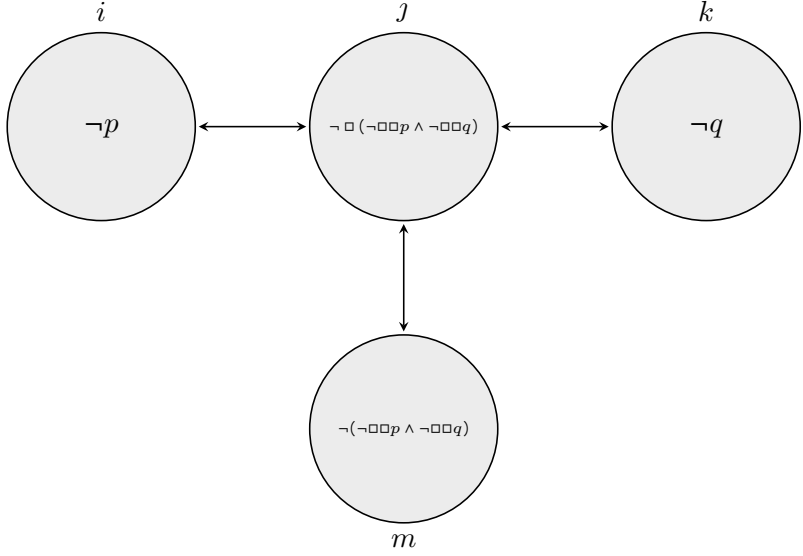
\begin{center}
\begin{tikzpicture}[modal,world/.append style={minimum size=2.5cm}]

\node[world] (x) [label=above: {$i$}] {$\neg p$};

\node[world] (y) [label=above: {$j$},right=of x] {\tiny$\neg\Box(\neg{\Box\Box} p\land\neg{\Box\Box} q)$};

\node[world] (z) [label=above: {$k$},right=of y] {$\neg q$};

\node[world] (w) [label=below: {$m$},below=of y] {\tiny$\neg(\neg{\Box\Box} p\land\neg{\Box\Box} q)$};

\path[<->] (x) edge (y); \path[<->] (y) edge (z); \path[<->] (y) edge (w); 
\end{tikzpicture}\end{center}
We have two possibilities: one where $v({\Box\Box} p,m)=1$ and another where $v({\Box\Box} q,m)=1$. In the former, $v(\Box p,j)=v(i,p)=1\neq 0=v(p,i)$. In the latter, $v(\Box q,j)=v(q,k)=1\neq 0=v(q,k)$. In each case, a contradiction results. So, there can be no $\mathbf{KB}$ countermodel. 
	
\end{proof}

\begin{lemma}
\label{lem:Bvalid}
	$\vDash_\mathbf{B} J=\Rightarrow p\sslash\Rightarrow \Box(\neg{\Box\Box} p\land \neg{\Box\Box} q)\sslash\Rightarrow q$.
\end{lemma}

\begin{proof}

Every $\mathbf{B}$ model is also a $\mathbf{KB}$ model. So, if there were an $\mathbf{B}$ countermodel to $J$, there would be a $\mathbf{KB}$ countermodel also. From \cref{lem:{KB}valid} there is no $\mathbf{KB}$ countermodel. Hence, there isn't a $\mathbf{B}$ one either.

\end{proof}

\begin{figure}[h]
\caption{Reflexivity-like Rules}
\label{fig:ECMT}
\centering
\begin{tabular}{lll}

 \AxiomC{$G\sslash\Gamma\Rightarrow\Delta\sslash\Gamma\Rightarrow\Delta\sslash G'$}
\RightLabel{\scriptsize{$EC$}}
\UnaryInfC{$G\sslash\Gamma\Rightarrow\Delta\sslash G'$}
\DisplayProof

&

 \AxiomC{$G\sslash\Gamma_1\Rightarrow\Delta_1\sslash\Gamma_2\Rightarrow\Delta_2\sslash G'$}
\RightLabel{\scriptsize{$Merge$}}
\UnaryInfC{$G\sslash\Gamma_1,\Gamma_2\Rightarrow\Delta_1,\Delta_2\sslash G'$}
\DisplayProof

\\\\
\end{tabular}

\begin{tabular}{ccccccccccccc}

 \AxiomC{$G\sslash\Gamma,\phi\Rightarrow\Delta\sslash G'$}
\RightLabel{\scriptsize{$T$}}
\UnaryInfC{$G\sslash\Gamma,\Box\phi\Rightarrow\Delta\sslash G'$}

\DisplayProof
\end{tabular}
\end{figure}
What follows has two parts. One part involves showing that the rule $EC$ is admissible in $\mathrm{RTB}$, the system that is just like $\mathrm{RB}$ except that it has the rule $T$ as a basic rule instead of $EC$ (see \Cref{fig:ECMT}). The other part involves showing that $J$ is unprovable in $\mathrm{RKB}$ and $\mathrm{RTB}$. We do the latter first and then show that $EC$ is admissible in $\mathrm{RTB}$. It then follows that $J$ is unprovable in $\mathrm{RB}$. 

\subsection{$J$ is unprovable in $RKB$ and $RTB$}

$\mathrm{RKB}$ and $\mathrm{RB}$ are known to be sound relative to $\mathbf{KB}$ and $\mathbf{B}$ Kripke frames respectively \cite{BurnsForthcoming,ParisiForthcoming}. To this we add that $\mathrm{RTB}$ is sound relative to $\mathbf{B}$ Kripke frames.

\begin{lemma}
\label{lem:RTBsound}
If $\vdash_{\mathrm{RTB}} H$ then $\vDash_{\mathbf{B}} H$. 
\end{lemma}

\begin{proof}
The proof proceeds by induction on the length of derivations. We only display the $T$ case. We have a derivation $\delta$ of the following form:

\begin{tabular}{ccccccccccccc}
\\
 \AxiomC{$\vdots$}
 \noLine
 \UnaryInfC{$G\sslash\Gamma,\phi\Rightarrow\Delta\sslash G'$}
\RightLabel{\scriptsize{$T$}}
\UnaryInfC{$G\sslash\Gamma,\Box\phi\Rightarrow\Delta\sslash G'$}

\DisplayProof
\\\\
\end{tabular}

\noindent Suppose we have a countermodel $\mathfrak{M}$ to the endhypersequent. This is a branch of points $w_1,...,w_i,...w_n$ with $w_i$ being a countermodel to the displayed sequent. By the reflexivity condition on $\mathbf{B}$ frames, $w_iRw_i$. Hence, by the $\Box$ truth conditions, $v(\phi,w_i)=1$. This means that our branch of points is also a countermodel to the premise hypersequent. 

\end{proof}

\begin{lemma}
\label{lem:RTBinvalid}
	$\nvdash_{\mathrm{RTB}} J=\Rightarrow p\sslash\Rightarrow \Box(\neg{\Box\Box} p\land \neg{\Box\Box} q)\sslash\Rightarrow q$.
\end{lemma}

\begin{proof}

We perform a simple backwards proof search in $\mathrm{RTB}$. For the sake of \textit{reductio}, suppose we have a proof $\delta$ of $J$. The last rule applied in $\delta$ would need to either be $T\mathrm{R}$, with a subcase for each formula, or $Sym$.\\
Consider the first case, with three subcases. The subproof $\delta'$ of $\delta$ ending immediately before the application of $TR$ would be of one of the following forms:
\begin{itemize}
\item [(i)] \AxiomC{$\vdots\delta'$}
\noLine
\UnaryInfC{$\Rightarrow \sslash\Rightarrow \Box(\neg{\Box\Box} p\land \neg{\Box\Box} q)\sslash\Rightarrow q$}
\DisplayProof
\item [(ii)] \AxiomC{$\vdots\delta'$}
\noLine
\UnaryInfC{$\Rightarrow p\sslash\Rightarrow \sslash\Rightarrow q$}
\DisplayProof
\item [(iii)]\AxiomC{$\vdots\delta'$}
\noLine
\UnaryInfC{$\Rightarrow p\sslash\Rightarrow \Box(\neg{\Box\Box} p\land \neg{\Box\Box} q)\sslash\Rightarrow $}
\DisplayProof
\end{itemize} 
Each of (i)--(iii) has a simple $\mathbf{B}$ countermodel. Hence, via \Cref{lem:RTBsound} if $J$ were $\mathrm{RTB}$ provable, the last step could not be an application of $T\mathrm{R}$.\\
Consider now the second case of $Sym$. The subproof $\delta'$ of $\delta$ ending immediately before the application of $Sym$ would be of the following form ending in the hypersequent $J'$:
\begin{itemize}
\item  [(iv)] \AxiomC{$\vdots\delta'$}
\noLine
\UnaryInfC{$J'=\Rightarrow q\sslash\Rightarrow \Box(\neg{\Box\Box} p\land \neg{\Box\Box} q)\sslash\Rightarrow p$}
\DisplayProof
\end{itemize}
The last step of $\delta'$ could also, like that of $\delta$, either be $T\mathrm{R}$, with a subcase for each formula, or $Sym$. The former case is near identical to that of $\delta$, with simple $\mathbf{B}$ countermodels to each possible $T\mathrm{R}$ predecessor of $J'$.\\ 
Consider the second $Sym$ case. If the last step of $\delta'$ were $Sym$, then the endhypersequent of the subproof $\delta''$ ending immediately before the application of $Sym$ would be $J$ itself. But the proof couldn't just be endless iterations of $Sym$! Eventually there would have to be an application of $T\mathrm{R}$. Yet we've just seen that none of the possible $T\mathrm{R}$ predecessors of $J$ nor its converse $J'$ are $\mathrm{RTB}$ provable. Hence, $J$ is not $\mathrm{RTB}$ provable.

\end{proof} 

\begin{lemma}
\label{lem:RKBinvalid}
$\nvdash_{\mathrm{RKB}}J$
\end{lemma}

\begin{proof}

This is an immediate corollary of \Cref{lem:RTBinvalid}.

\end{proof}

\subsection{$EC$ is Admissable in $\mathrm{RTB}$}

We now show that $EC$ is admissible in $\mathrm{RTB}$. The first step in this is to show that $EC$ is derivable from the rule $Merge$. 

\begin{tabular}{lll}

\AxiomC{$\vdots$}
\noLine
\UnaryInfC{$G\sslash\Gamma\Rightarrow\Delta\sslash\Gamma\Rightarrow\Delta\sslash G'$}
\RightLabel{\scriptsize{$EC$}}
\UnaryInfC{$G\sslash\Gamma\Rightarrow\Delta\sslash G'$}
\DisplayProof

&

\AxiomC{$\vdots$}
\noLine
\UnaryInfC{$G\sslash\Gamma_1\Rightarrow\Delta_1\sslash\Gamma_2\Rightarrow\Delta_2\sslash G'$}
\RightLabel{\scriptsize{$Merge$}}
\UnaryInfC{$G\sslash\Gamma_1,\Gamma_2\Rightarrow\Delta_1,\Delta_2\sslash G'$}
\DisplayProof

\end{tabular}

\begin{lemma}
$EC$ is derivable from $Merge$
\label{lem:ECMerge}
\end{lemma}

\begin{proof}
\AxiomC{$\vdots$}
\noLine
\UnaryInfC{$G\sslash\Gamma\Rightarrow\Delta\sslash\Gamma\Rightarrow\Delta\sslash G'$}
\RightLabel{\scriptsize{$Merge$}}
\UnaryInfC{$G\sslash\Gamma,\Gamma\Rightarrow\Delta,\Delta\sslash G'$}
\DisplayProof\\
As our sequents are pairs of sets, the conclusion hypersequent is identical to $G\sslash\Gamma\Rightarrow\Delta\sslash G'$.
\end{proof}

We next show that $Merge$ is admissible in $\mathrm{RTB}$. For this we use the following modified definitions of a \emph{main sentence} and \emph{main sequent} from \cite[p.88]{Parisi2017}.  

\begin{definition}[Main Sentence]
$\phi$ is the main sentence of $T\mathrm{L}$, $T\mathrm{R}$, $\neg \mathrm{L}$, $\neg \mathrm{R}$, $\land \mathrm{L}$, $\land \mathrm{R}$, $\Box \mathrm{L}$, $\Box \mathrm{R}$ or $T$, if $\phi$ appears in the conclusion of one of those but not the premise(s). In $\mathrm{Id}$, the main sentence is the only sentence present.
\end{definition}

\begin{definition}[Main Sequent]
The main sequent of a rule schema is either the sequent containing the main sentence of the rule, or is given by the following list:
\begin{itemize}
\item The main sequent of $EW\mathrm{L}$ and $EW\mathrm{R}$ is $\Rightarrow$; and
\item For $\Box \mathrm{L}$ and $\Box \mathrm{R}$,  the \emph{left-main} sequent is the one containing the main sentence. The \emph{right-main} sequent is either the one immediately following the \emph{left-main} sequent or none at all.\footnote{The order of left and right have been changed from \cite[p.88]{Parisi2017} because of the difference in notation.}
\end{itemize}
\end{definition}

\begin{lemma}
If $\vdash_{\mathrm{RTB}} G\sslash\Gamma_1\Rightarrow\Delta_1\sslash\Gamma_2\Rightarrow\Delta_2\sslash G'$ then $\Rightarrow_{\mathrm{RTB}} G\sslash\Gamma_1,\Gamma_2\Rightarrow\Delta_1,\Delta_2\sslash G'$
\label{lem:RTBmerge}
\end{lemma} 

\begin{proof}
The proof proceeds by induction on the length of derivations. The base case is trivial as the antecedent of the lemma does not hold. For the induction step, there are a lot of cases to check. We display those for the $\Box$ rules and $\land \mathrm{R}$. The remainder are left to the reader.
In each case we have a derivations $\delta, \delta^+$  of length $n$ of the premise of the rule, which is extended to a derivation $\delta'$ of length $n+1$. We show that if \cref{lem:RTBmerge} holds of the conclusion of $\delta$ then it also holds of the conclusion of $\delta'$.
\begin{itemize}
\item [Case 1:] $\Box \mathrm{L}$. We have four subcases, depending on whether $\Gamma_1\Rightarrow\Delta_1$ or $\Gamma_2\Rightarrow\Delta_2$ are left- or right-main. We display the two subcases where, first $\Gamma_1\Rightarrow\Delta_1$ is left-main and $\Gamma_2\Rightarrow\Delta_2$ is right-main, and second $\Gamma_1\Rightarrow\Delta_1$ is right-main.
\begin{itemize}
\item [Subcase 1:] $\Gamma_1\Rightarrow\Delta_1$ is left-main (and  $\Gamma_2\Rightarrow\Delta_2$ is right-main). We have a derivation $\delta'$ of the form

\AxiomC{$\vdots\delta$}
\noLine
\UnaryInfC{$G\sslash\Gamma_1\Rightarrow\Delta_1\sslash\Gamma_2,\phi\Rightarrow\Delta_2\sslash G'$}
\RightLabel{\scriptsize{$\Box \mathrm{L}$}}
 \UnaryInfC{$G\sslash\Gamma_1,\Box\phi\Rightarrow\Delta_1\sslash\Gamma_2\Rightarrow\Delta_2\sslash G'$}
\DisplayProof\\

We apply the induction hypothesis to $\delta$ and then apply $T$, giving us the desired result

\AxiomC{$\vdots IH\delta$}
\noLine
\UnaryInfC{$G\sslash\Gamma_1,\phi,\Gamma_2\Rightarrow\Delta_1,\Delta_2\sslash G'$}
\RightLabel{\scriptsize{$T$}}
 \UnaryInfC{$G\sslash\Gamma_1,\Box\phi,\Gamma_2\Rightarrow\Delta_1,\Delta_2\sslash G'$}
\DisplayProof

\item [Subcase 2:] $\Gamma_1\Rightarrow\Delta_1$ is right-main. We have a derivation $\delta'$ of the form
\AxiomC{$\vdots\delta$}
\noLine
\UnaryInfC{$G\sslash\Sigma\Rightarrow\Lambda\sslash\Gamma_1,\phi\Rightarrow\Delta_1\sslash\Gamma_2\Rightarrow\Delta_2\sslash G'$}
\RightLabel{\scriptsize{$\Box \mathrm{L}$}}
 \UnaryInfC{$G\sslash\Sigma,\Box\phi\Rightarrow\Lambda\sslash\Gamma_1\Rightarrow\Delta_1\sslash\Gamma_2\Rightarrow\Delta_2\sslash G'$}
\DisplayProof\\

We apply the induction hypothesis to $\delta$ and then apply $\Box \mathrm{L}$, giving us the desired result

\AxiomC{$\vdots IH\delta$}
\noLine
\UnaryInfC{$G\sslash\Sigma\Rightarrow\Lambda\sslash\Gamma_1,\phi,\Gamma_2\Rightarrow\Delta_1,\Delta_2\sslash G'$}
\RightLabel{\scriptsize{$\Box \mathrm{L}$}}
 \UnaryInfC{$G\sslash\Sigma,\Box\phi\Rightarrow\Lambda\sslash\Gamma_1,\Gamma_2\Rightarrow\Delta_1,\Delta_2\sslash G'$}
\DisplayProof

The remaining subcases are like Subcase 2 in that $\Box \mathrm{L}$ is used after applying the induction hypothesis, rather than $T$.
\end{itemize}

\item[Case 2:] $\Box \mathrm{R}$. We have two subcases, one where $\Gamma_2\Rightarrow\Delta_2$ is left-main and another where neither of $\Gamma_1\Rightarrow_\Delta1$ or $\Gamma_2\Rightarrow\Delta_2$ is left-main nor right-main. We display the former and leave the latter to the reader (the reasoning is essentially the same). We have a derivation $\delta'$ of the form

\AxiomC{$\vdots\delta$}
\noLine
\UnaryInfC{$G\sslash\Gamma_1\Rightarrow\Delta_1\sslash\Gamma_2\Rightarrow\Delta_2\sslash\Rightarrow\phi$}
\RightLabel{\scriptsize{$\Box \mathrm{R}$}}
 \UnaryInfC{$G\sslash\Gamma_1\Rightarrow\Delta_1\sslash\Gamma_2,\Rightarrow\Delta_2,\Box\phi$}
\DisplayProof

We then assume the induction hypothesis of $\delta$ and apply $\Box \mathrm{R}$: 

\AxiomC{$\vdots  IH\delta$}
\noLine
\UnaryInfC{$G\sslash\Gamma_1,\Gamma_2\Rightarrow\Delta_1,\Delta_2\sslash\Rightarrow\phi$}
\RightLabel{\scriptsize{$\Box \mathrm{R}$}}
 \UnaryInfC{$G\sslash\Gamma_1,\Gamma_2\Rightarrow\Delta_1,\Delta_2,\Box\phi$}
\DisplayProof

\item [Case 3:] $\land \mathrm{R}$. We have three subcases. Two where $\Gamma_1\Rightarrow\Delta_1$ and $\Gamma_2\Rightarrow\Delta_2$ respectively are main sequents and another where neither are. We display the first .

\AxiomC{$\vdots\delta$}
\noLine
\UnaryInfC{$G\sslash\Gamma_1\Rightarrow\Delta_1,\phi\sslash\Gamma_2\Rightarrow\Delta_2\sslash G'$}
\AxiomC{$\vdots\delta^+$}
\noLine
\UnaryInfC{$G\sslash\Gamma_1\Rightarrow\Delta_1,\psi\sslash\Gamma_2\Rightarrow\Delta_2\sslash G'$}
\RightLabel{\scriptsize{$\land \mathrm{R}$}}
 \BinaryInfC{$G\sslash\Gamma_1\Rightarrow\Delta_1,\phi\land\psi\sslash\Gamma_2\Rightarrow\Delta_2\sslash G'$}
\DisplayProof

We then assume the induction hypothesis to $\delta$ and $\delta^+$ and apply $\land \mathrm{R}$:

\AxiomC{$\vdots  IH\delta$}
\noLine
\UnaryInfC{$G\sslash\Gamma_1,\Gamma_2\Rightarrow\Delta_1,\phi,\Delta_2\sslash G'$}
\AxiomC{$\vdots  IH\delta^+$}
\noLine
\UnaryInfC{$G\sslash\Gamma_1,\Gamma_2\Rightarrow\Delta_1,\psi,\Delta_2\sslash G'$}
\RightLabel{\scriptsize{$\land \mathrm{R}$}}
 \BinaryInfC{$G\sslash\Gamma_1,\Gamma_2\Rightarrow\Delta_1,\phi\land\psi,\Delta_2\sslash G'$}
\DisplayProof
\end{itemize}
\end{proof}

\subsection{Incompleteness}

We now put the results of the previous two sections together to prove cut-free hypersequent incompleteness of $\mathrm{RB}$.

\begin{lemma}
$\nvdash_{\mathrm{RB}} J$
\label{lem:RBinvalid}
\end{lemma}

\begin{proof}
Given $EC$ is admissible in $\mathrm{RTB}$ (\Cref{lem:ECMerge} and \cref{lem:RTBmerge}), anything provable in $\mathrm{RB}$ is provable in $\mathrm{RTB}$. But we know that $J$ is unprovable in $\mathrm{RTB}$ (Lemma \ref{lem:RTBinvalid}). Hence $J$ is unprovable in $\mathrm{RB}$.
\end{proof}

\begin{theorem}
$\mathrm{RB}$ is cut-free hypersequent incomplete relative to $\mathbf{B}$ (symmetric, reflexive) Kripke frames.
\end{theorem}

\begin{proof}
$\vDash_\mathbf{B} J$ (\Cref{lem:Bvalid})  but $\nvdash_{\mathrm{RB}} J$ (\cref{lem:RBinvalid}). 
\end{proof}

\begin{theorem}
$\mathrm{RKB}$ is cut-free hypersequent incomplete relative to $\mathbf{KB}$ (symmetric) Kripke frames.
\end{theorem}

\begin{proof}
$\vDash_{\mathbf{KB}} J$ (\cref{lem:{KB}valid})  but $\nvdash_{\mathrm{RKB}} J$ (\Cref{lem:RKBinvalid}). 
\end{proof}

It would be nice to know whether $\mathrm{RKB}$ and $\mathrm{RB}$ are also cut-free formula and sequent incomplete. Unfortunately this will remain open in the current paper.

\section{Contrasting the $4$ and $\mathrm{B}$ results}
\label{sec:contrast}

We have managed to show formula, sequent and hypersequent cut-free incompleteness for $\mathrm{RK4}$ and $\mathrm{RS4}$, whereas we have only managed to show hypersequent cut-free incompleteness for $\mathrm{RKB}$ and $\mathrm{RB}$. In the former two, formula, sequent and hypersequent incompleteness directly hang together. For, in both $\mathrm{RK4}$ and $\mathrm{RS4}$, given a hypersequent $H$, there is a formula $I(H)$ such that $H$ is provable iff the hypersequent $\Rightarrow I(H)$ is provable. In $\mathrm{RKB}$ and $\mathrm{RB}$, however, it is unclear whether given a hypersequent $H$ there is such an equivalent formula. 

In the $\mathrm{RK4}$ and $\mathrm{RS4}$ cases, we use a translation from hypersequents to formulas from Burns and Zach
 \cite[p.6]{BurnsForthcoming}: 

{\centering
\begin{tabular}{l}
\\
$I(\Gamma\Rightarrow\Delta)=\bigwedge\Gamma\rightarrow\bigvee\Delta$\\
$I(\Gamma\Rightarrow\Delta\sslash H)= (\bigwedge\Gamma\rightarrow\bigvee\Delta)\lor\Box I(H)$
\\\\
\end{tabular}

}
\noindent where  $\phi\rightarrow\psi \coloneq  \neg\phi\lor\psi$.
Burns and Zach show the equivalence of a relational hypersequent $H$ and its formula translation $I(H)$ within Kripke frames \cite[p.6]{BurnsForthcoming}.\footnote{The same kind of reasoning can be used to show that they are equivalent in the PS4 models from \Cref{def:PS4M}.} The counterexample $C$ to $\mathrm{RK4}$ and $\mathrm{RS4}$'s cut-free completeness is the formula translation of the hypersequent  $\Box\neg\Box(p\land q)\Rightarrow \sslash\Box p\Rightarrow \sslash\Box q\Rightarrow$. That this hypersequent was a counterexample to $\mathrm{RK4}$ and $\mathrm{RS4}$'s hypersequent cut-free completeness was found first and then formula (and therefore sequent) cut-free incompleteness was found via the formula translation. For in $\mathrm{RK4}$ and $\mathrm{RS4}$ a hypersequent $H$ and its formula translation are also equivalent. In the lead up to proving the equivalence, we state the following reduction lemmas.

\begin{lemma}

For both $\mathrm{RK4}$ and $\mathrm{RS4}$: 
\begin{enumerate}
\item If $\vdash H\sslash \Gamma\Rightarrow\Delta,\phi\lor\psi\sslash G$ then  $\vdash H\sslash \Gamma\Rightarrow\Delta,\phi,\psi\sslash G$;
\item If $\vdash H\sslash \Gamma,\phi\land\psi\Rightarrow\Delta\sslash G$ then  $\vdash H\sslash \Gamma,\phi,\psi\Rightarrow\Delta\sslash G$ ;
\item If $\vdash H\sslash \Gamma\Rightarrow\Delta,\neg\phi\sslash G$ then  $\vdash H\sslash \Gamma,\phi\Rightarrow\Delta\sslash G$; 
\item If $\vdash H\sslash \Gamma\Rightarrow\Delta,\Box\phi\sslash G$ then  $\vdash H\sslash \Gamma\Rightarrow\Delta\sslash\Rightarrow \phi$. 

\end{enumerate}
\label{lem:invert}
\end{lemma}

\begin{proof}
The proofs are routine inductions on the length of derivations. We display the $\land \mathrm{R}$ case for (4) as an example. 

\begin{itemize}

\item [Case 1 $\land \mathrm{R}$:]  

\begin{tabular}{lll}
&
\AxiomC{$\vdots\delta$}
\noLine
\UnaryInfC{$H\sslash \Gamma\Rightarrow \psi,\Delta,\Box\phi\sslash G$}
\AxiomC{$\vdots\delta'$}
\noLine
\UnaryInfC{$H\sslash \Gamma\Rightarrow \xi,\Delta,\Box\phi\sslash G$}
\RightLabel{\scriptsize{$\land\mathrm{R}$}}
\BinaryInfC{$H\sslash \Gamma\Rightarrow \psi\land\xi,\Delta,\Box\phi\sslash G$}
\DisplayProof

\\\\
$\Longrightarrow$
&

\AxiomC{$\vdots$IH$\delta$}
\noLine
\UnaryInfC{$H\sslash \Gamma\Rightarrow \psi,\Delta\sslash\Rightarrow\phi$}
\AxiomC{$\vdots$IH$\delta'$}
\noLine
\UnaryInfC{$H\sslash \Gamma\Rightarrow \xi,\Delta\sslash\Rightarrow\phi$}
\RightLabel{\scriptsize{$\land\mathrm{R}$}}
\BinaryInfC{$H\sslash \Gamma\Rightarrow \psi\land\xi,\Delta\sslash\Rightarrow\phi$}
\DisplayProof
\\
\end{tabular}

\noindent We simply assume IH of $\delta$ and $\delta'$ and then apply $\land \mathrm{R}$.

\end{itemize}

\end{proof}

\begin{theorem}
For both $\mathrm{RK4}$ and $\mathrm{RS4}$: $\vdash H$ iff $\vdash \Rightarrow I(H)$
\end{theorem}

\begin{proof}
For the proof we have the two hypersequents:
\begin{itemize}
\item $H=\Gamma_1\Rightarrow\Delta_1\sslash...\sslash\Gamma_n\Rightarrow\Delta_n$; and
\item $I(H)= (\bigwedge\Gamma_1\rightarrow\bigvee\Delta_1) \lor\Box((...\Box(\bigwedge\Gamma_n\rightarrow\bigvee\Delta_n)...)) $.
\end{itemize}
For the the left to right direction we assume a derivation of $H$. It is simply a matter of applying the connective rules to derive $I(H)$ from $H$.\\
For the right to left direction we assume a derivation $\delta\vdash I(H)$. We proceed by induction on $n$ as follows: 
In the base $n=1$ case, where $H=\Gamma_1\Rightarrow\Delta_1$ and $I(H)=\bigwedge\Gamma_1\rightarrow\bigvee\Delta_1$, you just apply \Cref{lem:invert}(1)-(3) in whichever order you want to obtain the fact that there is a derivation $\delta'\vdash H=\Gamma_1\Rightarrow\Delta_1$.\\ 
For the induction step we have the instances:
\begin{itemize}
\item$H=\Gamma_1\Rightarrow\Delta_1\sslash H'$; and
  \item $I(H)=(\bigwedge\Gamma_{1}\rightarrow\bigvee\Delta_{1})\lor\Box I(H')$
\end{itemize}
where $H'$ is $n$ sequents long. Applying first \Cref{lem:invert}(1) and then (4) shows us that there is a derivation $\delta'\vdash(\bigwedge\Gamma_1\rightarrow\bigvee\Delta_1)\sslash I(H')$. We then apply the same reasoning as in the base case to show that there is a derivation $\delta''\vdash(\Gamma_1\Rightarrow\Delta_1)\sslash I(H')$. From the induction hypothesis we have that there is a derivation $\delta'''\vdash(\Gamma_1\Rightarrow\Delta_1)\sslash H'$.
\end{proof}

In contrast to $\mathrm{RK4}$ and $\mathrm{RS4}$, in $\mathrm{RKB}$ and $\mathrm{RB}$ a hypersequent and its Burns and Zach formula translation are not always equivalent. The left to right direction of the equivalence does hold -- the reasoning simply involves applying the relevant connective rules to the hypersequent $H$, just as with $\mathrm{RK4}$ and $\mathrm{RS4}$. The equivalence breaks down, however, in the right to left direction. For $J$ is unprovable, whereas $\Rightarrow I(J)$ is provable. Consider the following proof of $I(J)$: 

\begin{tabular}{c}
\small
\AxiomC{$p\Rightarrow p$}
\RightLabel{\scriptsize{$EW\mathrm{L}$}}
\UnaryInfC{$\Rightarrow\sslash p\Rightarrow p$}
\RightLabel{\scriptsize{$EW\mathrm{L}$}}
\UnaryInfC{$\Rightarrow\sslash\Rightarrow\sslash p\Rightarrow p$}
\RightLabel{\scriptsize{$\Box \mathrm{L}$}}
\UnaryInfC{$\Rightarrow\sslash{\Box} p\Rightarrow\sslash\Rightarrow p$}
\RightLabel{\scriptsize{$\Box \mathrm{L}$}}
\UnaryInfC{${\Box\Box} p\Rightarrow\sslash\Rightarrow\sslash\Rightarrow p$}
\RightLabel{\scriptsize{$Sym$}}
\UnaryInfC{$\Rightarrow p\sslash\Rightarrow\sslash{\Box\Box} p\Rightarrow $}
\RightLabel{\scriptsize{$\neg \mathrm{R}$}}
\UnaryInfC{$\Rightarrow p\sslash\Rightarrow\sslash\Rightarrow\neg{\Box\Box} p $}
\RightLabel{\scriptsize{$T \mathrm{R}$}}
\UnaryInfC{$\Rightarrow p\sslash\Rightarrow\Box q\sslash\Rightarrow\neg{\Box\Box} p $}

\AxiomC{$q\Rightarrow q$}
\RightLabel{\scriptsize{$EW\mathrm{L}$}}
\UnaryInfC{$\Rightarrow\sslash q\Rightarrow q$}
\RightLabel{\scriptsize{$\Box \mathrm{L}$}}
\UnaryInfC{${\Box} q\Rightarrow\sslash\Rightarrow q$}
\RightLabel{\scriptsize{$\Box \mathrm{R}$}}
\UnaryInfC{${\Box} q\Rightarrow\Box q$}
\RightLabel{\scriptsize{$EW\mathrm{L}$}}
\UnaryInfC{$\Rightarrow\sslash\Box q\Rightarrow\Box q$}
\RightLabel{\scriptsize{$\Box \mathrm{L}$}}
\UnaryInfC{${\Box\Box} q\Rightarrow\sslash\Rightarrow{\Box} q$}
\RightLabel{\scriptsize{$Sym$}}
\UnaryInfC{$\Rightarrow{\Box} q\sslash{\Box\Box} q\Rightarrow$}
\RightLabel{\scriptsize{$\neg \mathrm{R}$}}
\UnaryInfC{$\Rightarrow{\Box} q\sslash\Rightarrow\neg{\Box\Box} q$}
\RightLabel{\scriptsize{$EW\mathrm{L}$}}
\UnaryInfC{$\Rightarrow\sslash\Rightarrow{\Box} q\sslash\Rightarrow\neg{\Box\Box} q$}
\RightLabel{\scriptsize{$T \mathrm{R}$}}
\UnaryInfC{$\Rightarrow p\sslash\Rightarrow{\Box} q\sslash\Rightarrow\neg{\Box\Box} q$}

\RightLabel{\scriptsize{$\land \mathrm{R}$}}
\BinaryInfC{$\Rightarrow p\sslash\Rightarrow\Box q\sslash\Rightarrow\neg{\Box\Box} p\land\neg{\Box\Box} q$}
\RightLabel{\scriptsize{$\Box \mathrm{R}$}}
\UnaryInfC{$\Rightarrow p\sslash\Rightarrow\Box q,\Box(\Rightarrow\neg{\Box\Box} p\land\neg{\Box\Box} q)$}
\RightLabel{\scriptsize{$\lor \mathrm{R}_1$}}
\UnaryInfC{$\Rightarrow p\sslash\Rightarrow\Box q,\Box(\neg{\Box\Box} p\land\neg{\Box\Box} q)\lor\Box q$}
\RightLabel{\scriptsize{$\lor \mathrm{R}_2$}}
\UnaryInfC{$\Rightarrow p\sslash\Box(\neg{\Box\Box} p\land\neg{\Box\Box} q)\lor\Box q$}
\RightLabel{\scriptsize{$\Box \mathrm{R}$}}
\UnaryInfC{$\Rightarrow p,\Box(\Box(\neg{\Box\Box} p\land\neg{\Box\Box} q)\lor\Box q)$}
\RightLabel{\scriptsize{$\lor \mathrm{R}_2$}}
\UnaryInfC{$\Rightarrow p,p\lor\Box(\Box(\neg{\Box\Box} p\land\neg{\Box\Box} q)\lor\Box q)$}
\RightLabel{\scriptsize{$\lor \mathrm{R}_1$}}
\UnaryInfC{$\Rightarrow p\lor\Box(\Box(\neg{\Box\Box} p\land\neg{\Box\Box} q)\lor\Box q)$}

\DisplayProof
\\\\
\end{tabular}

\noindent\textit{If} there is an adequate formula translation of hypersequents for $\mathrm{RKB}$ and $\mathrm{RB}$, i.e. a mapping $I'$ such that a hypersequent $H$ is provable iff the hypersequent $\Rightarrow I'(H)$ is provable, then we have a quick route to formula and, therefore also, sequent cut-free incompleteness. Whether there is one remains to be found.\footnote{We also have a breakdown of the equivalence of $\Box(\phi\land\psi)$ and $\Box\phi\land\Box\psi$. We do have $\vdash_{\mathrm{RKB}}\Box\phi\land\Box\psi\Rightarrow\Box(\phi\land\psi)$ and $\vdash_{\mathrm{RKB}}\Box(\phi\land\psi)\Rightarrow\Box\phi\land\Box\psi$. Interestingly, however, while $J$ is unprovable, the hypersequent $J'=\Rightarrow p//\Rightarrow\Box\neg{\Box\Box} p\land\Box\neg{\Box\Box} q//\Rightarrow q$ is provable (the proof is very similar to that of $I(J)$). This is also a concrete example of $\mathrm{Cut}$ failing, because if $\mathrm{Cut}$ were admissible, from the proof of $J'$ we would know that there is a proof of $J$. }

\section{Conclusions and Open Questions}
\label{sec:end}

This paper has answered a number of questions that were raised in \Cref{sec:c}. We now know that $\mathrm{RKB}$, $\mathrm{RB}$, $\mathrm{RK4}$ and $\mathrm{RS4}$ are cut-free hypersequent incomplete, and that the latter two are also cut-free sequent and formula incomplete. Hence, the cut-free systems are not adequate for the intended Kripke frames. Importantly, as a consequence $\mathrm{Cut}$ is not an admissible rule in any of the four systems, causing problems for Parisi's project of using them as the basis for an inferentialist account of modality.  
This still leaves open the cut-free sequent and formula completeness, and hence adequacy, of $\mathrm{RKB}$ and $\mathrm{RB}$. They may turn out to be cut-free formula and sequent complete, even if though they are cut-free hypersequent incomplete. If so, it's conceivable that someone might be more concerned about the former than the latter. 
After all, sequent completeness captures the notion of being complete in regards to arguments and formula completeness in regards to theorems. In contrast, there isn't a pre-existing notion that hypersequent completeness captures, plausibly because hypersequents have been introduced as a tool for obtaining an adequate proof theory. 
However, if, with Parisi, one accepts that Cut Admissibility is required for an inferentialist account of modality, then Parisi's systems $\mathrm{RK4}$, $\mathrm{RS4}$, $\mathrm{RKB}$ and $\mathrm{RB}$ will not do, questions of completeness aside. 

A number of other questions remain, some technical, others more philosophical:

\begin{itemize}

\item Are $\mathrm{RKB}$ and $\mathrm{RB}$ cut-free formula and sequent complete? 

\item Are there adequate hypersequent systems that meet Parisi's, and Burns and Zach's criteria, i.e.\ cut-admissibility, Do\v{s}en's principle, and modularity? What is common in the cases discussed in this paper is that the tree structure of standard Kripke frames is not fully captured in Parisi's relational hypersequents, at least for $\mathrm{RK4}$, $\mathrm{RS4}$, $\mathrm{RKB}$ and $\mathrm{RB}$.  There might be a way to capture this with relational hypersequents using different rules. Alternatively, a more complex structure like Poggiolesi's tree hypersequents \cite{Poggiolesi2011} might be needed.

\item What kind of models are Parisi's cut-free $\mathrm{RKB}$, $\mathrm{RB}$, $\mathrm{RK4}$ and $\mathrm{RS4}$ complete relative to? The latter two may be complete relative to the Pseudo S4 models defined in this paper. Conversely, what is the logic of the Pseudo-Models? While cooked up for the purpose proving incompleteness, they may be worth studying in their own right.

\end{itemize}

\appendix
\section{$\mathrm{RK4}$ and $\mathrm{RS4}$ with Cut are Hypersequent Complete}

For clarity, we refer to $\mathrm{RK4}$ and $\mathrm{RS4}$ with $\mathrm{Cut}$ as a basic rule as $\mathrm{RK4_{Cut}}$ and $\mathrm{RS4_{Cut}}$ respectively. 

The following is a modification of Burns and Zach's cut-free completeness proofs \cite{BurnsForthcoming}. Rather than reproduce the proof in total, only the modifications are given here, with the reader directed to the relevant parts of \cite{BurnsForthcoming}. These are all from \S 3 of \cite{BurnsForthcoming} unwards. 

Definitions 13-15 are left unchanged. We modify Definition 16 \cite[p.10, 15]{BurnsForthcoming} to replace Burns and Zach's reduction rule for $\Box \mathrm{L}$ with $\Box \mathrm{L'}$.

{\raggedleft
\begin{tabular}{l|c|c|r}

$\Box\mathrm{L}$ & $G//\Box\phi,\Gamma'\overset{\sigma'}{\Rightarrow}\Delta'//\Gamma\overset{\sigma}{\Rightarrow}\Delta//G'$ & $G//\Box\phi,\Gamma'\overset{\sigma'}{\Rightarrow}\Delta'//\phi,\Gamma\overset{\sigma}{\Rightarrow}\Delta//G'$\\

$\Box\mathrm{L}'$ & $G//\Gamma,\Box\phi\overset{\sigma'}{\Rightarrow}\Delta//G'//\Sigma\overset{\sigma}{\Rightarrow}\Lambda//G''$ & $G//\Gamma,\Box\phi\overset{\sigma'}{\Rightarrow}\Delta//G'//\Sigma,\phi\overset{\sigma}{\Rightarrow}\Lambda//G''$
\\\\
\end{tabular}
Note that $\Box \mathrm{L}$ is an instance of $\Box \mathrm{L'}$. 
}

The proof of Proposition 17 \cite[p.11-12]{BurnsForthcoming} is then modified to show that $\Box \mathrm{L'}$ preserves unprovability in $\mathrm{RK4_{Cut}}$ and $\mathrm{RS4_{Cut}}$. This is shown by the following derivation: 

{\small
{\raggedleft
\AxiomC{$\vdots$}
\noLine
\UnaryInfC{$G\sslash\Gamma\Rightarrow\Delta\sslash G'\sslash\Sigma,\phi\Rightarrow\Lambda\sslash G''$}
\RightLabel{\scriptsize{$T\mathrm{L}$}}
\UnaryInfC{$G\sslash\Gamma,\Box\phi\Rightarrow\Delta\sslash G'\sslash\Sigma,\phi\Rightarrow\Lambda\sslash G''$}

\AxiomC{$\phi\Rightarrow\phi$}

\RightLabel{\scriptsize{$EW\mathrm{L}$}}
\UnaryInfC{$\Rightarrow\sslash \phi\Rightarrow\phi$}
\RightLabel{\scriptsize{$\Box\mathrm{L}$}}
\UnaryInfC{$\Box\phi\Rightarrow\sslash\Rightarrow\phi$}
\RightLabel{\scriptsize{$EW$}}
\UnaryInfC{$\Box \phi\Rightarrow\sslash \Rightarrow\sslash\Rightarrow\phi$}
\RightLabel{{*}}
\UnaryInfC{$G\sslash\Gamma,\Box\phi\Rightarrow\Delta\sslash G'\sslash\Sigma\Rightarrow\Lambda,\phi\sslash G''$}
\RightLabel{\scriptsize{$\mathrm{Cut}$}}
\BinaryInfC{$G\sslash\Gamma,\Box\phi\Rightarrow\Delta\sslash G'\sslash\Sigma\Rightarrow\Lambda\sslash G''$}
\DisplayProof
\\
* Multiple possible applications of internal and external weakening

}
}

We leave Proposition 18 and its proof unchanged. However, we modify Proposition 19 \cite[p.11, 15]{BurnsForthcoming} to add on further components Proposition 19(6) and (7):
\begin{itemize}
\item[(6)] If $\Box\phi\in\Gamma,\sigma{R^+}\tau$ and $\tau$ occurs in $H$, then $\phi\in\Gamma(H,\tau)$.
\item[(7)] If $\Box\phi\in\Gamma,\sigma{R^*}\tau$ and $\tau$ occurs in $H$, then $\phi\in\Gamma(H,\tau)$.
\end{itemize}
\begin{proof}
For (6), suppose that $\Box\phi\in\Gamma$, $\sigma{R^+}\tau$ and $\tau$ occurs in $H$. Since $\Sigma(H)$ is an $R^1$-branch and $\sigma{R^+}\tau$, the component $H(\tau)$ occurs to the right of $H(\sigma)$. Because $\Box\phi\in\Gamma$, the hypersequent $G\sslash\Gamma,\Box\phi\overset{\sigma}\Rightarrow\Delta\sslash G'\sslash\Sigma,\phi\overset{\tau}{\Rightarrow}\Lambda\sslash G''$ is a $\Box \mathrm{L'}$ $\tau$-reduct of $G\sslash\Gamma,\Box\phi\overset{\sigma}\Rightarrow\Delta\sslash G'\sslash\Sigma\overset{\tau}{\Rightarrow}\Lambda\sslash G''$. Since $H$ is $\tau$-reduced, $H$ is identical to all its $\Box \mathrm{L'}$ $\tau$-reducts. Therefore, $\phi\in\Gamma(H,\tau)$.\\
(7) follows from (5') and (6).  
\end{proof}
The following definitions 20 and 22, and propositions 21 and 23 are left unchanged. 

Lastly, we change the definition of a model used in Proposition 24 to use $R^+$ in the $\mathrm{RK4_{Cut}}$ case and $R^*$ in the $\mathrm{RS4_{Cut}}$ case, and employ our additions to Proposition 19, (6) and (7) in the proof.  For $\mathrm{RK4_{Cut}}$ in the proof of Proposition 24, in the case where $\Box\phi\in\Gamma(\sigma)$ we use Proposition 19(7) instead of Proposition 19(5). For $\mathrm{RK4_{Cut}}$ in the proof of the case where $\Box\phi\in\Gamma(\sigma)$, we use Proposition 19(7).

\begin{theorem}
\label{th:RK4CutComplete}
If $\vDash_{\mathbf{K4}} H$ then $\vdash_{\mathrm{RK4_{Cut}}} H$
\end{theorem}

\begin{proof}
This follows from the modified proof of Proposition 24 above, setting the accessibility relation to $R^+$. 
\end{proof}

\begin{theorem}
\label{th:RS4CutComplete}
If $\vDash_{\mathbf{S4}} H$ then $\vdash_{\mathrm{RS4_{Cut}}} H$
\end{theorem}

\begin{proof}
This follows from the modified proof of Proposition 24 above, setting the accessibility relation to $R^*$. 
\end{proof}

\small
\bibliographystyle{plain}
\bibliography{Cut-Free_Hypersequent_S4_and_RB_are_Hypersequent_Incomplete}

\end{document}